\documentclass[a4paper,12pt]{article}

\usepackage{amsmath}
\usepackage{amsthm}
\usepackage{amssymb}

\usepackage{epsfig}
\usepackage{psfrag}
\usepackage{subfigure}
\usepackage{graphicx}

\usepackage[mathscr,mathcal]{eucal}
\usepackage{enumerate}

\usepackage{t1enc}
\usepackage[latin2]{inputenc}

\def \C {\mathbb C}

\def \N {\mathbb N}

\def \R {\mathbb R}

\def \Z {\mathbb Z}

\def\cB{\mathcal{B}}
\def\cC{\mathcal{C}}
\def\cD{\mathcal{D}}

\def\cF{\mathcal{F}}

\def\cH{\mathcal{H}}

\def\cK{\mathcal{K}}
\def\cL{\mathcal{L}}

\def\cS{\mathcal{S}}

\def\vareps{\varepsilon}

\def\vp{\mathbf{p}}

\def\vx{\mathbf{x}}
\def\vy{\mathbf{y}}

\newcommand{\expect}[1]{\ensuremath{\mathbf{E}\big(#1\big)}}

\newcommand{\condprob}[2]{\ensuremath{\mathbf{P}\big(#1\bigm|#2\big)}}
\newcommand{\condexpect}[2]{\ensuremath{\mathbf{E}\big(#1\bigm|#2\big)}}

\newcommand{\ind}[1]{\ensuremath{{1\!\!1}{\big\{\,#1\,\big\}}}}

\def\la{\langle}
\def\ra{\rangle}

\def\one{1\!\!1}

\def\grad{\mathrm{grad}\,}

\DeclareMathOperator*{\Perm}{Perm}

\renewcommand{\d}{\mathrm d}

\newcommand{\abs}[1]{\left|\,{#1}\,\right|}
\newcommand{\norm}[1]{\left\|\,{#1}\,\right\|}

\def \wt {\widetilde}

\def\wh{\widehat}

\newtheorem {theorem}{Theorem}

\newtheorem {lemma}{Lemma}

\newtheorem {corollary}{Corollary}

\newtheorem* {theorem*}{Theorem}
\newtheorem* {thm*}{Theorem}
\newtheorem* {lemma*}{Lemma}
\newtheorem* {lem*}{Lemma}
\newtheorem* {corollary*}{Corollary}
\newtheorem* {cor*}{Corollary}
\newtheorem* {proposition*}{Proposition}
\newtheorem* {prop*}{Proposition}
\newtheorem* {definition*}{Definition}
\newtheorem* {def*}{Definition}
\newtheorem* {conjecture*}{Conjecture}
\newtheorem* {remark*}{Remark}
\newtheorem* {rem*}{Remark}

\def\be{\begin{equation}}
\def\ee{\end{equation}}

\def\bea{\begin{eqnarray}}
\def\eea{\end{eqnarray}}

\newcommand{\wick}[1]{\ensuremath{:\!\! #1 \!\!:\,}}

\title{Diffusive limit for \\ self-repelling Brownian polymers in $d\ge3$}

\author{
{\sc Ill\'es Horv\'ath} \qquad {\sc B\'alint T\'oth} \qquad {\sc B\'alint Vet\H o}
\\
Institute of Mathematics, Budapest University of Technology
\\
Egry J\'ozsef u.\ 1, Budapest, H-1111, Hungary
\\
email: {\tt \{pollux,balint,vetob\}@math.bme.hu}
}

\begin{document}

\maketitle

\begin{abstract}
The \emph{self-repelling Brownian polymer} model (SRBP) initiated by Durrett and Rogers in \cite{durrett_rogers_92} is the continuous space-time counterpart of the \emph{myopic (or 'true') self-avoiding walk} model (MSAW) introduced in the physics literature by Amit, Parisi and Peliti in \cite{amit_parisi_peliti_83}. In both cases, a random motion in space is pushed towards domains less visited in the past by a kind of negative gradient of the occupation time measure.

We investigate the asymptotic behaviour of SRBP in the non-recurrent dimensions. First, extending 1$d$ results from \cite{tarres_toth_valko_09}, we identify a natural stationary (in time) and ergodic distribution of the environment (essentially, smeared-out occupation time measure of the process), as seen from the moving particle. As main result we prove that in three and more dimensions, in this stationary (and ergodic) regime, the displacement of the moving particle scales diffusively and its finite dimensional distributions
converge to those of a Wiener process. This result settles part of the conjectures (based on non-rigorous renormalization group arguments) in \cite{amit_parisi_peliti_83}.

The main tool is the non-reversible version of the Kipnis\,--\,Varadhan-type CLT for additive functionals of ergodic Markov processes and the \emph{graded sector condition} of Sethuraman, Varadhan and Yau, \cite{sethuraman_varadhan_yau_00}.


\end{abstract}

\section{Introduction and background}
\label{s:intro}


The asymptotic scaling behaviour of \emph{self-repelling} random motions with long memory has been a mathematical challenge since the early eighties. The two basic models considered in the physical and probabilistic literature are the so-called \emph{myopic (or 'true') self-avoiding random walk} (MSAW), which appeared first in the physics literature in \cite{amit_parisi_peliti_83}, and the so-called \emph{self-repelling Brownian polymer} (SRBP) model, which was
initiated in the probabilistic literature in \cite{norris_rogers_williams_87}, \cite{durrett_rogers_92}. The two models (or, better said, families of models), although having their origins in different cultures and having different motivations, are phenomenologically very similar.

The simplest formulation of the MSAW model is as follows: Let $X(n)$ be a nearest neighbour random walk on $\Z^d$ and
\begin{equation}
\ell(n,y):=\ell(0,y)+\abs{\{0<m\le n: X(m)=y\}}
\end{equation}
its occupation time measure, taken with some (possibly signed) initial values $\ell(0,y)\in\Z$. The walk is governed by the following law:
\begin{align}
\label{law}
&
\condprob{X(n+1)=y}{\text{past, }X(n)=x}=
\\
\notag & \hskip5cm \ind{\abs{x-y}=1}
\frac{r(\ell(n,x)-\ell(n,y))}{\sum_{z:\abs{z-x}=1} r(\ell(n,x)-\ell(n,z))}
\end{align}
where $r:\Z\to(0,\infty)$ is a non-decreasing (and non-constant) weight function. (In the original \cite{amit_parisi_peliti_83} and the subsequent
physics papers, the specific choice $r(u)=\exp\{\beta u\}$, $\beta>0$, was made.) This means phenomenologically that the random walk $X(n)$ is pushed by the negative gradient of its occupation time measure.

The SRBP model is defined as follows. Let $V:\R^d\to\R_+$ be an approximate identity, that is a smooth ($C^{\infty}$), spherically symmetric function with sufficiently fast decay at infinity, and
\begin{equation}
\label{FisgradV}
F:\R^d\to\R^d,
\qquad
F(x):=-\grad \, V(x).
\end{equation}
For reasons which will be clarified later, we also impose the condition of \emph{positive definiteness} of $V$:
\begin{equation}
\label{Vposdef}
\wh V(p):=(2\pi)^{-d/2}\int_{\R^d} e^{i p\cdot x}V(x) \,\d x
\ge0.
\end{equation}
A particular choice could be $V(x):=\exp\{-\abs{x}^2/2\}$.

Let $t\mapsto B(t)\in\R^d$ be standard $d$-dimensional Brownian motion and define the stochastic process $t\mapsto X(t)\in\R^d$ as the solution of the SDE
\begin{equation}
\label{Brpoly} X(t)=  B(t)+\int_0^t\int_0^s F(X(s)-X(u))\,\d u\,\d s,
\end{equation}
or
\begin{equation}
\label{Brpolydiff}
\d X(t)= \d B(t)+
\big(\int_0^t F(X(t)-X(u))\,\d u\big)\,\d t.
\end{equation}

\medskip
\noindent{\bf Remark:}
Other types of self-interaction functions $F$ give rise to various different asymptotics. For the few rigorous results (mostly in 1d), see \cite{norris_rogers_williams_87}, \cite{durrett_rogers_92}, \cite{cranston_lejan_95}, \cite{cranston_mountford_96} and in particular
\cite{mountford_tarres_08} which also contains a survey of the results. Recent 1d results appear in \cite{tarres_toth_valko_09}.

\medskip

Now, introducing the occupation time measure
\begin{equation}
\ell(t, A):=\ell(0,A) + \abs{\{0<s\le t: X(s)\in A\}}
\end{equation}
where $A\subset \R^d$ is any measurable domain, and $\ell(0,A)$ is some signed initialization, we can rewrite the SDE \eqref{Brpolydiff} as follows:
\begin{equation}
\label{Brpolydiff2} \d X(t)= \d B(t) - \grad \big(V*\ell
(t,\cdot)\big)(X(t))\,\d t
\end{equation}
where $*$ stands for convolution in $\R^d$. We assume that $\ell(0,A)$ is a signed Borel measure on $\R^d$ with slow increase: for any $\varepsilon>0$
\begin{equation}
\label{slowincreaseinitially}
\lim_{N\to\infty} N^{-(d+\vareps)}\abs{\ell}(0,[-N,N]^d)=0.
\end{equation}
The form \eqref{Brpolydiff2}, compared with \eqref{law}, shows explicitly the phenomenological similarity of the two models.

Non-rigorous (but nevertheless convincing) scaling and renormalization group arguments originally formulated for the MSAW, but equally well applicable to the SRBP suggest the following dimension-dependent asymptotic scaling behaviour
(see e.g.\ \cite{amit_parisi_peliti_83}, \cite{obukhov_peliti_83}, \cite{peliti_pietronero_87}):

\begin{enumerate}[--]
\item
in $d=1$: $X(t)\sim t^{2/3}$, with intricate, non-Gausssian scaling limit;
\item
in $d=2$: $X(t)\sim t^{1/2}(\log t)^\zeta$, with some controversy about the value of the exponent $\zeta$ in the logarithmic correction, and Gaussian (that is Wiener) scaling limit expected;
\item
in $d\ge3$: $X(t)\sim t^{1/2}$, with Gaussian (i.e.\ Wiener) scaling limit
expected.
\end{enumerate}

In $d=1$, for some particular cases of the MSAW model (MSAW with edge, rather than site repulsion and MSAW with continuous time and site repulsion), the limit theorem for $X(n)/n^{2/3}$ was established in \cite{toth_95}, respectively, \cite{toth_veto_09}, with the truly intricate limiting
distribution identified. The scaling limit of the \emph{process}  $t\mapsto X(nt)/n^{2/3}$ was constructed and analyzed in \cite{toth_werner_98}. The proofs in \cite{toth_95} and \cite{toth_veto_09} have some built-in combinatorial elements which make it difficult (if possible at all) to extend these proofs robustly to a full class of 1d models of random motions pushed by the negative gradient of their occupation time measure. However, more recently, a robust proof was given for the super-diffusive behaviour of the 1d models: in
\cite{tarres_toth_valko_09}, inter alia, it is proved that for the $1d$ SRBP models $\varliminf_{t\to\infty}t^{-5/4}\expect{X(t)^2}>0$,
$\varlimsup_{t\to\infty}t^{-3/2}\expect{X(t)^2}<\infty$. These are robust super-diffusive  bounds (not depending on microscopic details) but still far from the expected $t^{2/3}$ scaling.

In $d=2$, very little is proved rigorously. For a modified version of MSAW where self-repulsion acts only in one spatial (say, the horizontal) direction, the marginally super-diffusive lower bound $\varliminf_{t\to\infty} t^{-1}(\log t)^{-1/2} \expect{X(t)^2}>0$ holds, cf.\ \cite{valko_09}.

In the present paper, we address the $d\ge3$ case of the SRBP model. We identify a stationary and ergodic distribution of the environment as seen from the position of the moving point and in this particular stationary regime, we prove \emph{diffusive limit} (that is non-degenerate CLT with normal scaling) for the displacement. Our general approach is that of martingale approximation for additive functionals of ergodic Markov processes, initiated for reversible processes in the classic Kipnis\,--\,Varadhan paper \cite{kipnis_varadhan_86} and extended to non-reversible cases in \cite{toth_86}, \cite{varadhan_96}, \cite{sethuraman_varadhan_yau_00}. We shall refer to this approach as the \emph{Kipnis\,--\,Varadhan theory}. In particular, validity of the efficient martingale approximation will rely on checking the \emph{graded sector condition} of \cite{sethuraman_varadhan_yau_00}.

Similar results for the MSAW model on the lattice $\Z^d$, $d\ge3$, will be presented in \cite{horvath_toth_veto_10}

Next, we describe our results in plain words. For precise formulations, see subsection \ref{ss:setup_and_results}.

As a first step, we note that the environment profile appearing on the right-hand side of \eqref{Brpolydiff}, \eqref{Brpolydiff2}, as seen in a moving coordinate frame tied to the current position of the process, $t\mapsto\eta(t,\cdot)$:
\begin{align}
\label{envir}
\eta(t,x)
:=&\,\,
\eta(0,X(t)+x)+
\int_0^t V(X(t)+x-X(u))\,\d u
\\[2pt]
\notag
=&\,\,
\eta(0,X(t)+x)+\big(V*\ell (t,\cdot)\big)(X(t)+x)
\end{align}
is a Markov process in a properly chosen function space $\Omega$, to be specified later. As a first step, we identify a natural \emph{time-stationary and ergodic distribution} of this process \eqref{envir}. Rather surprisingly, this is the Gaussian (scalar) field $x\mapsto\omega(x)\in\R$ with expectation and covariances
\begin{equation}
\label{cov} \expect{\omega(x)}=0, \qquad C(x-y):= \expect{\omega(x)\omega(y)} =
g*V(x-y)
\end{equation}
where
\begin{equation}
\label{greenf} g:\R^d\to\R, \qquad g(x):= \abs{x}^{2-d}
\end{equation}
is the Green function of the Laplacian in $\R^d$. Note that throughout this paper $d\ge3$. This is the \emph{massless free Gaussian field} whose ultraviolet singularity is smeared out by convolution with the smooth and rapidly decaying approximate identity $V$. The Fourier transform of the covariance is
\begin{equation}
\label{covarFour}
\wh C(p)=\abs{p}^{-2} \wh V(p).
\end{equation}
See Theorem \ref{thm:stat_erg}. All further results will be meant for the process being in this stationary regime. From this result, by ergodicity, the law of large numbers for the process $X(t)$ drops out, see Corollary \ref{cor:lln}.

The main result of the paper refers to the \emph{diffusive limit} of the process $t\mapsto X(t)$. From \eqref{Brpoly} and \eqref{envir}, it arises that the displacement is written as
\begin{equation}
\label{displ} X(t)= B(t) + \int_0^t\varphi(\eta(s))\,\d s
\end{equation}
where $\varphi:\Omega\to\R^d$ is a function of the state of the stationary and ergodic Markov process $t\mapsto\eta(t)$:
\begin{equation}
\label{phidef}
\varphi(\omega)= -\grad \omega(0).
\end{equation}
So, the natural approach to the diffusive limit of $X(t)$ is the Kipnis\,--\,Varadhan theory. We will prove validity of an efficient martingale approximation by checking Sethuraman\,--\,Varadhan\,--\,Yau's \emph{graded sector condition}, cf.\ \cite{sethuraman_varadhan_yau_00}. It is easy to see that due to the spherical symmetry of the problem, we get
\begin{equation}
\label{sphe}
\expect{X_k(t)X_l(t)}
=
\delta_{k,l}d^{-1}\expect{\abs{X(t)}^2}.
\end{equation}
We prove that for $d\ge3$, the limiting variance
\begin{equation}
\label{variance}
\sigma^2:=d^{-1}\lim_{t\to\infty}t^{-1}\expect{\abs{X(t)}^2}\in(0,\infty)
\end{equation}
exists and the finite dimensional marginals of the diffusively rescaled process
\begin{equation}
\label{rescaled}
X_N(t)
:=
\frac{X(Nt)}{\sigma \sqrt N}
\end{equation}
converge to those of a standard $d$-dimensional Brownian motion. See Theorem \ref{thm:clt}. The main result shows similarity in spirit and techniques with those of \cite{komorowski_olla_03}, but the differences are also clear.

\bigskip

The results are meant \emph{in probability with respect to the initial profile} $\eta(0,x)$ sampled from the stationary (and ergodic) initial distribution hinted at above. Recent results by Cuny and Peligrad \cite{cuny_peligrad_09} raise the hope that the Kipnis\,--\,Varadhan theory could be enhanced to CLT
for \emph{almost all} initial conditions sampled according to the stationary distribution.

\bigskip

The rest of the paper is structured as follows: in section
\ref{s:setup_and_results}, we give the formal definitions, introduce notations, identify the stationary measure and formulate our main results precisely. In section \ref{s:spaces_and_operators}, we give the functional analytic
background: the Hilbert spaces and the (bounded and unbounded) linear operators involved and the infinitesimal generator will be presented. Ergodicity and LLN for the displacement drop out for free. The short section \ref{s:KV} is devoted to recalling the martingale approximation in Kipnis\,--\,Varadhan theory and the \emph{graded sector condition} of \cite{sethuraman_varadhan_yau_00}. Finally, in section \ref{s:proof}, we check the abstract functional analytic conditions for our particular problem, and we conclude with the proof of the CLT for the displacement.

\section{Formal setup and results}
\label{s:setup_and_results}


\subsection{The stationary measure}
\label{ss:stat_meas}

We start with the \emph{Ansatz} that the stationary distribution of the process $t\mapsto\eta(t,\cdot)$ is translation invariant zero mean Gaussian scalar  with some covariance
\begin{equation}
\label{ansatzcov}
\expect{\eta(t,x)\eta(t,y)}=C(y-x),
\end{equation}
to be identified at the end of the following computations.

In order to prove that this is indeed time-stationary, we have to show that for any test function $x\to u(x)\in\R$, the moment generating functional
\begin{equation}
\phi(t,u):=\expect{\exp\{\la u,\eta(t)\ra\}}
\end{equation}
is actually constant in time. In the present subsection, we use the notation
\begin{equation}
\la u, v\ra := \int_{\R^d} v(x)u(x)\,\d x.
\end{equation}

In the forthcoming computations of the present section all \emph{repeated subscripts} are summed from $1$ to $d$. Using \eqref{envir}, note that, by standard It\^o calculus,
\begin{align}
\label{ito} \d\la u, \eta(t)\ra = -  \la \partial_l u, \eta(t)\ra\,\d B_l(t) +
\frac12 \la \partial^2_{ll} u, \eta(t)\ra\,\d t - \la \partial_l u, \eta(t) \ra
\partial_l\eta(t,0) \,\d t + \la u, V\ra \,\d t.
\end{align}
Hence
\begin{align}
\label{condchar}
&
\condexpect{\d\exp\{\la u,\eta(t)\ra\} } {\cF_t}
=
\\[3pt]
\notag
&\qquad
\exp\{\la u,\eta(t)\ra\}
\left(
\frac12
\la \partial^2_{ll}u, \eta(t)\ra
+
\frac12
\la \partial_l u, \eta(t)\ra^2
-
\la \partial_l u, \eta(t) \ra \partial_l\eta(t,0)
+
\la u, V \ra \right)\d t.
\end{align}
Now, using the Ansatz that $x\mapsto\eta(t,x)$ (with $t$ fixed!) is a Gaussian field with covariance \eqref{ansatzcov}, by standard computations of Gaussian expectations, from \eqref{condchar}, we obtain
\begin{align}
\label{char}
\frac{\d \expect{\exp\{\la u,\eta(t)\ra\}}}{\d t}
=
&
\exp\{\la u, C*u\ra/2\}
\Big(
\frac12
\la \partial^2_{ll}u, C*u\ra
+
\frac12
\la \partial_l u, C*\partial_l u\ra
+
\\
\notag
&
+
\frac12
\la \partial_l u, C*u\ra^2
-
\la \partial_{l} u,\partial_{l} C\ra
-
\la \partial_l u, C*u\ra \la u, \partial_lC\ra
+
\la u, V\ra
\Big)\,\d t.
\end{align}
On the right-hand side the first two terms cancel out by an integration by parts. The third and fifth terms cancel one by one due to the simple fact that for any test function $u$,
\begin{equation}
\la \partial_l u, C*u\ra=0.
\end{equation}
Thus, the right-hand side of \eqref{char} is canceled out completely iff
\begin{equation}
\label{condition}
V = - \partial^2_{ll}C.
\end{equation}
This is equivalent to \eqref{cov}, \eqref{covarFour}. Note that $d\ge3$ was assumed.


\subsection{Formal setup and results}
\label{ss:setup_and_results}


\subsubsection{State space and Gaussian measure}

The proper state space of our basic processes will be the space of smooth scalar fields of slow increase at infinity:
\begin{equation}
\label{Omega}
\Omega := \big\{\omega\in C^{\infty}(\R^d\to\R)\,:\,
\norm{\omega}_{m,r}<\infty \big\}
\end{equation}
where $\norm{\omega}_{m,r}$ are the seminorms
\begin{equation}
\label{seminorms}
\norm{\omega}_{m,r} := \sup_{x\in\R^d} \,
\big(1+\abs{x}\big)^{-1/r} \, \abs{\partial^{\abs{m}}_{m_1,\dots,m_d}\omega(x)}
\end{equation}
defined for the multiindices $m=(m_1,\dots,m_d)$, $m_j\ge0$; and $r\ge1$. The space $\Omega$ endowed with these seminorms is a Fr\'echet space.

From Minlos's theorem (see \cite{simon_74}), it follows that there exists a unique Gaussian probability measure $\pi(\d\omega)$ on the space of tempered distributions $\cS^{\prime}(\R^d\to\R)$ with characteristic functional
\begin{equation}
\label{charfctnl}
\expect{\exp\{i\la u,\omega\ra\}}=
\exp\left\{-\frac12\int_{\R^d}\int_{\R^d} u(x)C(x-y)u(y)\,\d x \d y\right\},
\end{equation}
and from smoothness of the covariance function $C(x)$, it follows that the probability measure $\pi(\d\omega)$ is actually concentrated on the space
$\Omega\subset\cS^{\prime}(\R^d\to\R)$ and \eqref{cov} holds.

The Gaussian field $\omega(x)$ is realized e.g. as a moving average of white noise:
\begin{equation}
\omega(x)=\int_{\R^d}U(x-y)w(y)\,\d y,
\end{equation}
where $U$ is the unique positive definite function for which $U*U=V$ and $w$ is $d$-dimensional white noise.

The group of spatial translations
\begin{equation}
\label{shift}
\R^d\ni z\mapsto \tau_z:\Omega\to\Omega,
\qquad
(\tau_z\omega)(x):=\omega(x+z)
\end{equation}
acts naturally on $\Omega$ and preserves the probability measure $\pi(\d\omega)$. Actually, the dynamical system $(\Omega, \pi(\d\omega), \tau_z:z\in\R^d)$ is \emph{ergodic}.


\subsubsection{Processes}

First, we consider the process $t\mapsto(X(t),
\zeta(t,\cdot))\in\R^d\times\Omega$ defined as follows:
\begin{align}
\label{Xeq}
X(t)
&=
X(0) +  B(t) - \int_0^t\grad\zeta(s,X(s))\,\d s,
\\[5pt]
\label{zetaeq}
\zeta(t,x)
&=
\zeta(0,x)+\int_0^t V (x-X(s))\,\d s
\end{align}
where $t\mapsto B(t)$ is a standard $d$-dimensional Brownian motion, and $X(0)\in\R^d$, $\zeta(0,\cdot)\in\Omega$ are the initial data for the process $t\mapsto(X(t),\zeta(t,\cdot))$.

Written as a single SDE for the process $t\mapsto X(t)$, we get from \eqref{Xeq}, \eqref{zetaeq}:
\begin{equation}
\label{Xeq2}
X(t)
=
X(0) +  B(t) +
\int_0^t\left\{\zeta(0,X(s))+
\int_0^s F(X(s)-X(u))\,\d u\right\}\d s.
\end{equation}
The SDE \eqref{Xeq2} differs from the original SDE \eqref{Brpoly} only by the presence of the initial profile $\zeta(0,x)$, which is a natural modification.

From \eqref{Xeq} and \eqref{zetaeq}, it follows that
\begin{align}
\label{Xeqctd}
X(t_0+t)
&=
X(t_0) +  \big(B(t+t_0)-B(t_0)\big) -
\int_{t_0}^{t_0+t}\grad\zeta(s,X(s))\,\d s,
\\[5pt]
\label{zetaeqctd}
\zeta(t_0+t,x)
&=
\zeta(t_0,x)+\int_{t_0}^{t_0+t}
V(x- X(s))\,\d s.
\end{align}
From this form, it is apparent that the process
$t\mapsto(X(t),\zeta(t,\cdot))\in\R^d\times\Omega$ is Markovian.

The environment profile as seen from the moving point $X(t)$ is \begin{equation}
\label{etadef}
x\mapsto \eta(t,x):=\zeta(t, X(t)+x).
\end{equation}
From \eqref{Xeqctd}, \eqref{zetaeqctd}, we readily obtain that $t\mapsto \eta(t):=\eta(t,\cdot)$ is itself a Markov process on the state space $\Omega$.

We define the function $\varphi:\Omega\to\R^d$ by \eqref{phidef}, and from \eqref{Xeq}, \eqref{zetaeq}, \eqref{etadef}, we readily get \eqref{displ}.


\subsubsection{Results}

\begin{theorem}
\label{thm:stat_erg}
The Gaussian probability measure $\pi(\d\omega)$ on
$\Omega$, with mean $0$ and covariances \eqref{cov} is time-invariant and ergodic for the $\Omega$-valued Markov process $t\mapsto\eta(t)$.
\end{theorem}

\begin{corollary}
\label{cor:lln}
For $\pi$-almost all initial profiles $\zeta(0,\cdot)$,
\begin{equation}
\label{lln2}
\lim_{t\to\infty}\frac{X(t)}{t}=0
\quad
\mathrm{a.s.}
\end{equation}
\end{corollary}

\medskip
\noindent {\bf Remarks:} It is clear that, in dimensions $d\ge3$, other stationary distributions of the process $t\mapsto\eta(t)$ exist. In particular, due to transience of the process $t\mapsto X(t)$, the stationary measure (presumably) reached from starting with ``empty'' initial conditions $\eta(0,x)\equiv0$ certainly differs from our $\d\pi$. Our methods and results are valid for the particular stationary distribution $\d\pi$.

\bigskip

The main result of the present paper is the following theorem:

\begin{theorem}
\label{thm:clt} In dimensions $d\ge3$, the following hold:
\\
(i)
The limiting variance
\begin{equation}
\label{variance2}
\sigma^2:=d^{-1}\lim_{t\to\infty}t^{-1}\expect{\abs{X(t)}^2}
\end{equation}
exists and
\begin{equation}
\label{bounds}
1 \le \sigma^2 \le 1+\rho^2
\end{equation}
where
\begin{equation}
\label{sumcond}
\rho^2:= d^{-1}\int_{\R^d} \abs{p}^{-2}\wh V(p)\,\d p<\infty.
\end{equation}
(ii)
The finite dimensional marginal distributions of the diffusively rescaled process
\begin{equation}
\label{rescaled2}
X_N(t)
:=
\frac{X(Nt)}{\sigma \sqrt N}
\end{equation}
converge to those of a standard $d$-dimensional Brownian motion. The convergence is meant in probability with respect to the starting state $\eta(0)$ sampled according to $\d\pi$. 
\end{theorem}

Theorem \ref{thm:clt} will be proved by use of the martingale approximation of the Kipnis\,--\,Varadhan theory and the so-called \emph{graded sector condition} of \cite{sethuraman_varadhan_yau_00}.


\section{Spaces and operators}
\label{s:spaces_and_operators}

The natural formalism for the proofs of our theorems is that of Fock spaces and Gaussian Hilbert spaces, and linear operators over them. For basics of Gaussian Hilbert spaces and Wick products, see \cite{janson_97}, \cite{simon_74}. Our
main Hilbert space is $\cH:=\cL^2(\Omega,\pi)$. This is a Gaussian Hilbert space, and has very natural unitary equivalent representations as Fock spaces. We follow the usual notation of Euclidean quantum field theory, see e.g.\ \cite{simon_74}. In subsection \ref{ss:spaces}, we give formal definition of the three unitary equivalent representations of the Hilbert space
$\cL^2(\Omega,\pi)$. In subsection \ref{ss:operators}, we define the linear operators which are relevant for our purposes and we present their action on the three unitary equivalent formulations. In subsection \ref{ss:infgen}, the infinitesimal generator of the semigroup of the stationary Markov process
$t\mapsto\eta(t,\cdot)\in \Omega$, acting on $\cL^2(\Omega,\pi)$, and its adjoint is computed and the first consequences (ergodicity, LLN) are settled.


\subsection{Spaces}
\label{ss:spaces}

Throughout this paper, we use the convention of unitary Fourier transform
\begin{equation}
\label{FourierTransform}
\wh u(p):= (2\pi)^{-d/2}\int_{\R^d}e^{i p\cdot
x}u(x)\,\d x.
\end{equation}
and the shorthand notation
\begin{align}
\label{notation1}
&
\vx=(x_1,\dots,x_n)\in\big(\R^d\big)^n,
&&
x_m=(x_{m1},\dots,x_{md})\in\R^d,
&&
\partial_{ml}:=\frac{\partial}{\partial x_{ml}},
\\[5pt]
\label{notation3}
&
\vp=(p_1,\dots,p_n)\in\big(\R^d\big)^n,
&&
p_m=(p_{m1},\dots,p_{md})\in\R^d,
&&
\end{align}
$m=1,\dots,n,\,l=1,\dots,d$.

We denote by $\cS_n$, respectively, $\wh \cS_n$, the \emph{symmetric} Schwartz spaces
\begin{align}
\label{Sndef} \cS_n:= & \{u:\R^{dn}\to\C: u(\varpi\vx)=u(\vx),\,
\varpi\in\Perm(n)\},
\\[5pt]
\label{hSndef} \wh \cS_n:= & \{\wh u:\R^{dn}\to\C: \wh u(\varpi\vp)=\wh
u(\vp),\, \varpi\in\Perm(n)\}.
\end{align}
In the preceding formulas $\Perm(n)$ denotes the group of permutations on the $n$ indices.

The spaces $\cS_n$, respectively, $\wh \cS_n$ are endowed with the following scalar products
\begin{align}
\label{Knscprod}
&
\langle u,v\rangle:=
\int_{\R^{dn}}\int_{\R^{dn}}
\overline{u(\vx)}C(\vx-\vy)v(\vy) \,\d\vx\d\vy,
\\[5pt]
\label{hKnscprod}
&
\langle \wh u,\wh v\rangle:=
\int_{\R^{dn}}
\overline{\wh u(\vp)} \wh C(\vp) \wh v(\vp) \,\d\vp
\end{align}
where
\begin{equation}
\label{CnhCndef}
C(\vx-\vy):=\prod_{m=1}^n C(x_m-y_m),
\qquad
\wh C(\vp):=\prod_{m=1}^n \wh C(p_m).
\end{equation}
Let $\cK_n$ and $\wh \cK_n$ be the closures of $\cS_n$, respectively, $\wh \cS_n$ with respect to the Euclidean norms defined by these inner products. The Fourier transform \eqref{FourierTransform} realizes an isometric isomorphism
between the Hilbert spaces $\cK_n$ and $\wh \cK_n$.

These Hilbert spaces are actually the symmetrized $n$-fold tensor products
\begin{equation}
\label{KnhKn}
\cK_n:=\mathrm{symm}\big(\cK_1^{\otimes n}\big),
\qquad
\cK_n:=\mathrm{symm}\big(\wh\cK_1^{\otimes n}\big).
\end{equation}
Finally, the full Fock spaces are
\begin{equation}
\label{Fockspaces}
\cK:=\overline{\oplus_{n=0}^\infty \cK_n},
\qquad
\wh\cK:=\overline{\oplus_{n=0}^\infty \wh\cK_n}.
\end{equation}

The Hilbert space of our true interest is $\cH=\cL^2(\Omega,\pi)$. This is itself a graded Gaussian Hilbert space
\begin{equation}
\label{Hgraded} \cH=\overline{\oplus_{n=0}^\infty \cH_n}
\end{equation}
where the subspaces $\cH_n$ are isometrically isomorphic with the subspaces $\cK_n$ of $\cK$ through the identification
\begin{equation}
\label{HKisometry} \phi_n: \cK_n\to\cH_n, \quad
\phi_n(u):=\frac{1}{\sqrt{n!}}\int_{\R^{dn}}
u(\vx)\wick{\omega(x_1)\dots\omega(x_n)}\,\d\vx.
\end{equation}
Here and in the rest of this paper, we denote by \wick{X_1\dots X_n} the Wick product of the jointly Gaussian random variables $(X_1,\dots,X_n)$. In order to ease notation the mapping $\phi_1: \cK_1\to\cH_1$ will be simply denoted by $\phi$.

As the graded Hilbert spaces
\begin{equation}
\cH:=\overline{\oplus_{n=0}^\infty \cH_n}, \quad
\cK:=\overline{\oplus_{n=0}^\infty \cK_n}, \quad
\wh \cK:=\overline{\oplus_{n=0}^\infty \wh\cK_n}
\end{equation}
are isometrically isomorphic in a natural way, we shall move freely between the various representations.


\subsection{Operators}
\label{ss:operators}


\subsubsection{General notation}
\label{sss:general_notation}

We use the standard notation of Fock spaces. First, we give a general framework of notation and identities formulated over the Gaussian Hilbert space $\cH$. Then, we turn to our relevant linear operators and we give their representations in all three Hilbert spaces  $\cH$, $\cK$ and $\wh \cK$.

The action of linear operators over $\cH=\overline{\oplus_{n=0}^\infty\cH_n}$ will be typically given in terms of Wick monomials. It is understood that their
action is extended by linearity and graph closure.

Given a (bounded or unbounded) densely defined and closed linear operator $A$ over the basic Hilbert space $\cK_1$,  its second quantized version, acting over the graded Gaussian Hilbert space $\cH$ will be denoted by $\d\Gamma(A)$. This latter one acts over Wick monomials as follows,
$\d\Gamma(A):\cH_n\to\cH_n$,
\begin{equation}
\label{sqop}
\d\Gamma(A) \wick{\phi(v_1) \cdots \phi(v_n)} = \sum_{m=1}^n
\wick{\phi(v_1)\cdots \phi(Av_m) \cdots \phi(v_n)}.
\end{equation}

Given a vector $u$ from the basic Hilbert space $\cK_1$, the creation and annihilation (raising and lowering) operators associated to it, acting over the Gaussian Hilbert space $\cH$, will be denoted by $a^*(u):\cH_n\to\cH_{n+1}$, respectively, $a(u):\cH_n\to\cH_{n-1}$, acting on Wick monomials as
\begin{align}
\label{cropdef}
a^*(u)\wick{\phi(v_1)\dots\phi(v_n)}
&=\,\,\,
\wick{\phi(u)\phi(v_1)\dots\phi(v_n)},
\\[3pt]
\label{anopdef}
a(u)\wick{\phi(v_1)\dots\phi(v_n)}
&=\,\,\,
\sum_{m=1}^n\langle u,v_m\rangle
\wick{\phi(v_1) \dots \phi(v_{m-1}) \phi(v_{m+1}) \dots \phi(v_n)}.
\end{align}
For basics about creation, annihilation and second quantized operators, see e.g.\ \cite{simon_74} or \cite{janson_97}.

We also define the unitary involution $J$ on $\cH$:
\begin{equation}
\label{invop}
Jf(\omega):=f(-\omega),
\qquad
J\upharpoonright_{\cH_n}=(-1)^n I\upharpoonright_{\cH_n}.
\end{equation}

The well-known canonical commutation relations between the operators introduced are:
\begin{equation}
\label{ccr1}
[a(u),a(v)]=0,
\qquad
[a^*(u),a^*(v)]=0,
\qquad
[a(u),a^*(v)]=\langle u,v\rangle I,
\end{equation}
\begin{equation}
\label{ccr2} [\d\Gamma(A),a^*(u)]=a^*(Au), \qquad\qquad
[\d\Gamma(A),a(u)]=-a(A^*u),
\end{equation}
\begin{equation}
\label{Jcomm} [J,\d\Gamma(A)]=0, \qquad \{J,a^*(u)\}=0, \qquad \{J,a(u)\}=0.
\end{equation}

Two more operators will be needed: given an element $u\in\cK_1$, \emph{multiplication by} $\phi(u)$ will be denoted $M(u)$, that is, formally, for $f\in\cL^2(\Omega, \pi)$,
\begin{equation}
\label{mu}
\big(M(u) f\big)(\omega):= \phi(u)(\omega)f(\omega).
\end{equation}
Finally, for a fixed element $\vartheta\in\Omega$, we introduce \emph{differentiation in the direction} $\vartheta$: formally
\begin{equation}
\label{differ}
D_{\vartheta}f(\omega):=
\lim_{\varepsilon\to0}\varepsilon^{-1}
\big(f(\omega+\varepsilon\vartheta)-f(\omega)\big).
\end{equation}
Both operators are well-defined on Wick monomials, and are extended by linearity and graph closure.

Given $u\in\cK_1$ the identities \eqref{muidentity} and \eqref{differidentity} below hold:
\begin{enumerate}[(1)]
\item
The multiplication operator $M(u)$ is actually
\begin{equation}
\label{muidentity}
M(u)=a^*(u)+a(u).
\end{equation}
\item
If $C*u\in\Omega$ then
\begin{equation}
\label{differidentity}
D_{C*u}=a(u).
\end{equation}
\end{enumerate}
Both identities are checked by direct computation on Wick monomials. The identity \eqref{differidentity} is a particular case of the \emph{directional derivative} of Malliavin calculus, see \cite{janson_97}.


\subsubsection{Specific linear operators}
\label{sss:specific_operators}

The most relevant operators for our present purposes are
\begin{equation}
\label{operators1}
\nabla_l:=\d\Gamma(\partial_l),
\qquad
\Delta:=\sum_{l=1}^d\nabla_l^2,
\qquad
a_l:=a(\partial_l\delta_0),
\qquad
a^*_l:=a^*(\partial_l\delta_0)
\end{equation}
where $\partial_l=\frac{\partial}{\partial x_l}$ and $\delta_0$ is Dirac's delta concentrated on $0\in\R^d$. Note that $\delta_0$ and all its partial derivatives are in the Hilbert space $\cK_1$.

We give now their action on the spaces $\cH_n$, $\cK_n$ and $\wh\cK_n$. The point is that we are interested primarily in their action on the space $\cL^2(\Omega,\pi)=\overline{\oplus_{n=0}^\infty\cH_n}$, but explicit computations in later sections are handy in the unitary equivalent representations over the space $\wh\cK=\overline{\oplus_{n=0}^\infty\wh\cK_n}$. The action of various operators over $\cH_n$ will be given in terms of the Wick monomials $\wick{\omega(x_1)\dots\omega(x_n)}$ and it is understood that the operators are extended by linearity and graph closure.

\begin{itemize}

\item
The operators $\nabla_l$, $l=1,\dots,d$:
\begin{align}
\label{nablaonHn}
&
\nabla_l:\cH_n\to\cH_n,
&&
\nabla_l\wick{\omega(x_1)\dots\omega(x_n)}=
-\sum_{m=1}^n \wick{\omega(x_1)\dots\partial_l\omega(x_m)\dots\omega(x_n)},
\\[3pt]
\label{nablaonKn}
&
\nabla_l:\cK_n\to\cK_n,
&&
\nabla_l u(\vx)
=
\sum_{m=1}^n
\frac{\partial u}{\partial x_{ml}}(\vx),
\\[3pt]
\label{nablaonhKn}
&
\nabla_l:\wh\cK_n\to\wh\cK_n,
&&
\nabla_l \wh u(\vp)
=
i \big(\sum_{m=1}^n p_{ml}\big) \wh u(\vp).
\end{align}
Note that these are actually unbounded, closed, skew self-adjoint operators.  They are densely defined on $\cH_n$, $\cK_n$, respectively, $\wh\cK_n$.

\item
The operator $\Delta$:
\begin{align}
\label{DeltaonHn}
&
\Delta:\cH_n\to\cH_n,
&&
\Delta\wick{\omega(x_1)\dots\omega(x_n)}
=
\\[2pt]
\notag
&&&\hskip16mm
\sum_{l=1}^d\sum_{m,m'=1}^n
\wick{\omega(x_1) \dots \partial_{l}\omega(x_m) \dots \partial_{l}\omega(x_{m'}) \dots \omega(x_n)},
\\[3pt]
\label{DeltaonKn}
&
\Delta:\cK_n\to\cK_n,
&&
\Delta u(\vx)
=
\sum_{l=1}^d\sum_{m,m^{\prime}=1}^n
\frac{\partial^2 u}
{\partial x_{ml}\partial x_{m^{\prime}l}}(\vx),
\\[3pt]
\label{DeltaonhKn}
&
\Delta:\wh\cK_n\to\wh\cK_n,
&&
\Delta \wh u(\vp)
=
-\abs{\sum_{m=1}^n p_m}^2 \wh u(\vp).
\end{align}
The operator $\Delta$ is unbounded, densely defined, self-adjoint and positive. Note that $\Delta$ is \emph{not} the second quantized Laplacian.

\item
The operator $\abs{\Delta}^{-1/2}=(-\Delta)^{-1/2}$:
\begin{align}
\label{sqrtgreenoponHn}
&
\abs{\Delta}^{-1/2}:\cH_n\to\cH_n,
&&
\text{no explicit formula},
\\[3pt]
\label{sqrtgreenoponKn}
&
\abs{\Delta}^{-1/2}:\cK_n\to\cK_n,
&&
\text{no explicit formula},
\\[3pt]
\label{sqrtgreenoponhKn}
&
\abs{\Delta}^{-1/2}:\wh\cK_n\to\wh\cK_n,
&&
\abs{\Delta}^{-1/2} \wh u(\vp)
=
\abs{\sum_{m=1}^n p_m}^{-1} \wh u(\vp).
\end{align}
The operator $\abs{\Delta}^{-1/2}$ is unbounded, densely defined, self-adjoint and positive.

\item
The operators $\abs{\Delta}^{-1/2}\nabla_l$, $l=1,\dots,d$:
\begin{align}
\label{nonameopsonHn}
&
\abs{\Delta}^{-1/2}\nabla_l:\cH_n\to\cH_n,
&&
\text{no explicit formula},
\\[3pt]
\label{nonameopsonKn}
&
\abs{\Delta}^{-1/2}\nabla_l:\cK_n\to\cK_n,
&&
\text{no explicit formula},
\\[3pt]
\label{nonameopsonhKn} & \abs{\Delta}^{-1/2}\nabla_l:\wh\cK_n\to\wh\cK_n, &&
\abs{\Delta}^{-1/2}\nabla_l \wh u(\vp) = \frac{i\sum_{m=1}^n p_{ml}
}{\abs{\sum_{m=1}^n p_m}} \wh u(\vp).
\end{align}
These are \emph{bounded} skew self-adjoint operators with operator norm
\begin{equation}
\label{nonameopnorm}
\norm{ \abs{\Delta}^{-1/2}\nabla_l } =1.
\end{equation}

\item
The creation operators $a^*_l$, $l=1,\dots,d$:
\begin{align}
\label{astaronHn}
&
a^*_l:\cH_n\to\cH_{n+1},
&&
a^*_l\wick{\omega(x_1)\dots\omega(x_n)}
=
\wick{\partial_l\omega(0)\omega(x_1)\dots\omega(x_n)},
\\[5pt]
\label{astaronKn}
&
a^*_l:\cK_n\to\cK_{n+1},
&&
a^*_lu(x_1,\dots,x_{n+1})
=
\\[2pt]
\notag
&&&\hskip10mm
\frac{1}{\sqrt{n+1}}
\sum_{m=1}^{n+1}
\partial_l\delta(x_m) u(x_1,\dots, x_{m-1},x_{m+1},\dots, x_{n+1}) ,
\\[3pt]
\label{astaronhKn}
&
a^*_l:\wh\cK_n\to\wh\cK_{n+1},
&&
a^*_l\wh u(p_1,\dots,p_{n+1})
=
\\[2pt]
\notag
&&&\hskip18mm
\frac{1}{\sqrt{n+1}}
\sum_{m=1}^{n+1}
ip_{ml}\wh u(p_1,\dots, p_{m-1},p_{m+1},\dots, p_{n+1}).
\end{align}
The creation operators $a_l^*$, restricted to the subspaces $\cH_n$, $\cK_n$, respectively, $\wh\cK_n$ are bounded, with operator norm
\begin{equation}
\label{astaropnorm}
\norm{ a^*_l\upharpoonright_{\cH_n} \!\!\!\phantom{\Big|}}
=
\norm{ a^*_l\upharpoonright_{\cK_n} \!\!\!\phantom{\Big|}}
=
\norm{ a^*_l\upharpoonright_{\wh\cK_n} \!\!\!\phantom{\Big|}}
=
\sqrt{C(0)}
\sqrt{n+1}.
\end{equation}

\item
The annihilation operators $a_l$, $l=1,\dots,d$:
\begin{align}
\label{aonHn}
&
a_l:\cH_n\to\cH_{n-1},
&&
a_l\wick{\omega(x_1)\dots\omega(x_n)}
=
\\[2pt]
\notag &&& \hskip20mm \sum_{m=1}^n \partial_lC(x_m)
\wick{\omega(x_1)\dots\omega(x_{m-1}) \omega(x_{m+1}) \dots \omega(x_n)},
\\[5pt]
\label{aonKn} & a_l:\cK_n\to\cK_{n-1}, && a_lu(x_1,\dots,x_{n-1}) = \sqrt{n}
\int_{\R^d} u(x_1,\dots,x_{n-1},y) \partial_lC(y) \,\d y,
\\[3pt]
\label{aonhKn} & a_l:\wh\cK_n\to\wh\cK_{n-1}, && a_l\wh u(p_1,\dots,p_{n-1}) =
\sqrt{n} \int_{\R^d} \wh u(p_1,\dots,p_{n-1},q)iq_l\wh C(q) \,\d q.
\end{align}
The annihilation operators $a_l$ restricted to the subspaces $\cH_n$, $\cK_n$, respectively, $\wh\cK_n$ are bounded with operator norm
\begin{equation}
\label{astaropnorm}
\norm{ a_l\upharpoonright_{\cH_n}  \!\!\!\phantom{\Big|}}
=
\norm{ a_l\upharpoonright_{\cK_n} \!\!\!\phantom{\Big|}}
=
\norm{ a_l\upharpoonright_{\wh\cK_n} \!\!\!\phantom{\Big|}}
=
\sqrt{C(0)}\sqrt{n}.
\end{equation}
Furthermore, as the notation $a^*_l$ and $a_l$ suggests, these operators are adjoint of each other.
\end{itemize}

\bigskip

Since all computations will be performed in the representation $\wh\cK$, we give a common core for all the unbounded operators defined above -- and some others to appear in future sections:
\begin{equation}
\label{coredefin}
\wh\cC:=\oplus_{n=0}^\infty \wh\cC_n,
\qquad
\wh\cC_n:=
\{\wh u \in \wh\cK_n: \sup_{\vp\in\R^{dn}}\abs{\wh u(\vp)} < \infty \}.
\end{equation}
Note that the operator $\abs{\Delta}^{-1/2}$ is defined on the dense subspace $\wh\cC$  \emph{only for} $d\ge3$. Furthermore, in dimensions $d\ge3$, the operators $\abs{\Delta}^{-1/2}\upharpoonright_{\wh\cK_n}$ defined on the dense subspaces $\wh\cC_n$, are \emph{essentially self-adjoint}. This follows, e.g., from Propositions VIII.1, VIII.2 of \cite{reed_simon_vol1_80}.

Notice also that $\nabla$ is the infinitesimal generator of the \emph{unitary group of spatial translations} while $\Delta$ is the infinitesimal generator of the Markovian semigroup of \emph{diffusion in random scenery}
\begin{align}
\label{shiftgroup}
&
\exp\{z\nabla\}=T_z, \qquad
&&
T_zf(\omega):=f(\tau_z\omega),
\\[5pt]
\label{drscesemigroup}
&
\exp\{t\Delta\}=Q_t, \qquad
&&
Q_tf(\omega):=\int\frac{\exp\{-z^2/(2t)\}}{\sqrt{2\pi t}}
f(\tau_z\omega)\,\d z.
\end{align}


\subsection{The infinitesimal generator, stationarity,
\\ Yaglom reversibility, ergodicity}
\label{ss:infgen}

We denote by $P_t$ the semigroup of the process $\eta(t)$:
\begin{equation}
\label{semigroup}
P_t:\cH\to\cH,
\qquad
P_t f(\omega)
:=
\condexpect{f(\eta(t))}{\eta(0)=\omega}.
\end{equation}
Then $[0,\infty)\ni t\mapsto P_t\in\cB(\cH)$ is a Markovian contraction semigroup on $\cH$. In order to identify its infinitesimal generator, note that the infinitesimal change in the state of the Markov process $\eta(t)$ is due to the following three terms:
\begin{enumerate}[\quad(1)]\setlength{\itemsep}{0pt}
\item
infinitesimal spatial shift due to $\d B(t)$;
\item
infinitesimal spatial shift due to $-\grad\eta(t,0)\d t$;
\item
infinitesimal local change in $\eta$ due to increase of local time.
\end{enumerate}
Altogether
\begin{equation}
\label{ito2}
\eta(t+\d t,x)=
\eta(t, x+\d B(t) - \grad \eta(t,0) \d t) + V(x)\d t.
\end{equation}

Hence, given a sufficiently regular function on the state space $f:\Omega\to\R$, we compute
\begin{equation}
\label{infgencomp}
\lim_{t\to0}\frac{\condexpect{f(\eta(t)-f(\eta(0)))}{\eta(0)=\omega}}{t}
=
\left(
\frac12\Delta -
\sum_{l=1}^d M(\partial_l\delta_0) \nabla_l + D_{V}
\right)f(\omega).
\end{equation}
Recall \eqref{cov} and note that hence
\begin{equation}
V=-C*\sum_{l=1}^d\partial^2_{ll}\delta_0
\end{equation}
with
\begin{equation}
-\sum_{l=1}^d\partial^2_{ll}\delta_0\in\cK_1.
\end{equation}
Using \eqref{muidentity}, \eqref{ccr2} and \eqref{differidentity} (in this order), we readily obtain the following expression for the  infinitesimal generator of the semigroup $P_t$:
\begin{equation}
\label{infgen}
G
:=
\frac12\Delta +
\sum_{l=1}^d \big(a^*_l \nabla_l + \nabla_l a_l\big).
\end{equation}
This operator is well defined on Wick polynomials of the field $\omega(x)$ and is extended by linearity and graph closure. It is not difficult to see that it satisfies the criteria of the Hille\,--\,Yoshida theorem (see \cite{reed_simon_vol1_80}) and thus it is indeed the infinitesimal generator of a Markovian semigroup. We omit these technical details.

The adjoint generator is
\begin{equation}
\label{adjinfgen}
G^*
:=
\frac12\Delta -
\sum_{l=1}^d \big(a^*_l \nabla_l + \nabla_l a_l\big).
\end{equation}
Note that due to the inner coherence of the model the last two terms on the right hand side of \eqref{infgencomp} combine to give the tidy skew self-adjoint part of the infinitesimal generators in \eqref{infgen},
\eqref{adjinfgen}.

For later use, we introduce notation for the symmetric (self-adjoint) and anti-symmetric (skew self-adjoint) parts of the generator
\begin{align}
\label{symgen}
S
&:=
-\frac12(G+G^*)=  -\frac12 \Delta,
\\[5pt]
\label{Adecomp}
A
&:=
\phantom{-}\frac12(G-G^*)=
\sum_{l=1}^d \big(a^*_l \nabla_l + \nabla_l a_l\big)
=: A_++A_-.
\end{align}
It is a standard -- though not completely trivial -- exercise to check that the operators $S$ and $A$, a priori defined on the dense subspace $\wh\cC$ are indeed essentially self-adjoint, respectively, essentially skew self-adjoint.

Note that
\begin{equation}
\label{grading}
S:\cH_n\to\cH_n,
\quad
A_+:\cH_n\to\cH_{n+1},
\quad
A_-:\cH_n\to\cH_{n-1},
\quad
A_{\mp}=-A_{\pm}^*,
\end{equation}
and
\begin{equation}
\label{van_H0_H1}
S\upharpoonright_{\cH_0}=0,
\qquad
A_+\upharpoonright_{\cH_0}=0,
\qquad
A_-\upharpoonright_{\cH_0\oplus\cH_1}=0.
\end{equation}

It is clear that
\begin{equation}
\label{stateq}
G^*\one = 0,
\end{equation}
and hence, it follows that $\pi$ is indeed stationary distribution of the process $t\mapsto\eta(t)$ and $G^*$ is  itself the infinitesimal generator of the stochastic semigroup $P^*_t$ of the time-reversed process.

Actually, so-called \emph{Yaglom reversibility} holds. From \eqref{Jcomm}, it follows that
\begin{equation}
\label{yaglom}
G^{*}=JGJ.
\end{equation}
This identity means that the stationary forward process $(-\infty,\infty)\ni t\mapsto\eta(t)$ and the \emph{flipped backward process}
\begin{equation}
\label{revproc}
(-\infty,\infty)\ni t\mapsto\wt\eta(t):=-\eta(-t)
\end{equation}
obey the same law. This is a special kind of time-reversal symmetry called Yaglom reversibility, see \cite{yaglom_47}, \cite{yaglom_49}, \cite{dobrushin_suhov_fritz_88}.

Proving ergodicity is easy: The Dirichlet form of the process $t\mapsto\eta(t)$ is
\begin{equation}
\label{df}
\cD(f):=
-(f,G f)=
-\frac12 (f, \Delta f)=
\frac12 \sum_{l=1}^d\norm{\nabla_l f}^2.
\end{equation}
So,
\begin{equation}
\label{erg}
\big\{ \cD(f)=0 \big\}
\ \iff\
\big\{ \nabla_l f=0,\ l=1,\dots, l \big\}
\ \iff\
\big\{ f=\text{const.} \ \pi\text{-a.s.} \big\},
\end{equation}
since $z\mapsto\tau_z$ acts ergodically on $(\Omega,\pi)$.

This proves Theorem \ref{thm:stat_erg}. Corollary \ref{cor:lln} follows directly from \eqref{displ} by the ergodic theorem.


\section{CLT for additive functionals of ergodic Markov processes, graded sector condition}
\label{s:KV}

In the present short section we recall the non-reversible version of the Kipnis\,--\,Varadhan CLT for additive functionals of ergodic Markov processes and the \emph{graded sector condition} of Sethuraman, Varadhan and Yau, \cite{sethuraman_varadhan_yau_00}.

Let $(\Omega, \cF, \pi)$ be a probability space: the state space of a \emph{stationary and ergodic} Markov process  $t\mapsto\eta(t)$. We put ourselves in the Hilbert space $\cH:=\cL^2(\Omega, \pi)$. Denote the \emph{infinitesimal generator} of the semigroup of the process by $G$, which is
a well-defined (possibly unbounded) closed linear operator on $\cH$. The adjoint generator $G^*$ is the infinitesimal generator of the semigroup of the reversed (also stationary and ergodic) process $\eta^*(t)=\eta(-t)$. It is assumed that $G$ and $G^*$ have a \emph{common core of definition} $\cC\subseteq\cH$. Let $f\in\cH$, such that $(f, \one) = \int_\Omega f\,\d\pi=0$. We ask about CLT/invariance principle for
\begin{equation}
\label{rescaledintegral}
N^{-1/2}\int_0^{Nt} f(\eta(s))\,\d s
\end{equation}
as $N\to\infty$.

We denote the \emph{symmetric} and \emph{anti-symmetric} parts of the generators $G$, $G^*$, by
\begin{equation}
S:=-\frac12(G+G^*),
\qquad
A:=\frac12(G-G^*).
\end{equation}
These operators are also extended from $\cC$ by graph closure and it is assumed that they are well-defined self-adjoint, respectively, skew self-adjoint operators
\begin{equation}
S^*=S\ge0, \qquad A^*=-A.
\end{equation}
Note that $-S$ is itself the infinitesimal generator of a Markov semigroup  on $\cL^2(\Omega,\pi)$, for which the probability measure $\pi$ is reversible (not just stationary). We assume that $-S$ is itself ergodic:
\begin{equation}
\label{Sergodic}
\mathrm{Ker}(S)=\{c1\!\!1 : c\in\C\}.
\end{equation}

We denote by $R_\lambda\in\cB(\cH)$ the resolvent of the semigroup $s\mapsto e^{sG}$:
\begin{equation}
R_\lambda
:=
\int_0^\infty e^{-\lambda s} e^{sG}\d s
=
\big(\lambda I-G\big)^{-1}, \qquad \lambda>0,
\end{equation}
and given $f\in\cH$ as above, we will use the notation
\begin{equation}
u_\lambda:=R_\lambda f.
\end{equation}

The following theorem yields the efficient martingale approximation of the additive functional \eqref{rescaledintegral}:

\begin{theorem*}[\bf KV]
\label{thm:kv}
With the notation and assumptions as before, if the following two limits hold in $\cH$
\begin{align}
\label{conditionA}
&
\lim_{\lambda\to0}
\lambda^{1/2} u_\lambda=0,
\\[5pt]
\label{conditionB}
&
\lim_{\lambda\to0} S^{1/2} u_\lambda=:v\in\cH,
\end{align}
then
\begin{equation}
\label{kv_variance}
\sigma^2:=2\lim_{\lambda\to0}(u_\lambda,f)\in[0,\infty),
\end{equation}
and there exists a zero mean, $\cL^2$-martingale $M(t)$, adapted to the filtration of the Markov process $\eta(t)$ with stationary and ergodic increments and variance
\begin{equation}
\expect{M(t)^2}=\sigma^2t
\end{equation}
such that
\begin{equation}
\label{kv_martappr} \lim_{N\to\infty} N^{-1} \expect{\big(\int_0^N
f(\eta(s))\,\d s-M(N)\big)^2} =0.
\end{equation}
In particular, if $\sigma>0$, then the finite dimensional marginal distributions of the rescaled process $t\mapsto \sigma^{-1} N^{-1/2}\int_0^{Nt}f(\eta(s))\,\d s$ converge to those of a standard $1d$ Brownian motion.
\end{theorem*}

\bigskip
\noindent
{\bf Remarks:}
\paragraph{(1)}
Conditions \eqref{conditionA} and \eqref{conditionB} of the theorem are jointly equivalent to the following
\begin{equation}
\label{conditionC}
\lim_{\lambda,\lambda'\to0}(\lambda+\lambda')(u_\lambda,u_{\lambda'})=0.
\end{equation}
Indeed, straightforward computations yield:
\begin{equation}
\label{A+B=C}
(\lambda+\lambda')(u_\lambda,u_{\lambda'}) =
\norm{S^{1/2}(u_\lambda-u_{\lambda'})}^2 + \lambda \norm{u_\lambda}^2 +
\lambda' \norm{u_{\lambda'}}^2.
\end{equation}

\paragraph{(2)}
The theorem is a generalization to non-reversible setup of the
celebrated Kipnis\,--\,Varadhan theorem, \cite{kipnis_varadhan_86}. To the best of our knowledge, the non-reversible formulation, proved with resolvent rather
than spectral calculus, appears first -- in discrete-time Markov chain, rather than continuous-time Markov process setup and with condition \eqref{conditionC} -- in \cite{toth_86} where it was applied, with bare hand computations, to obtain CLT for a particular random walk in random environment. Its proof follows the original proof of the  Kipnis\,--\,Varadhan theorem with the difference that spectral calculus is to be replaced by resolvent calculus.

\paragraph{(3)}
In continuous-time Markov process setup, it was formulated in
\cite{varadhan_96} and applied to tagged particle motion in non-reversible zero mean exclusion processes. In this paper, the \emph{(strong) sector condition} was formulated, which, together with an $H_{-1}$-bound on the function $f\in\cH$, provide sufficient condition for \eqref{conditionA} and
\eqref{conditionB} of Theorem KV to hold.

\paragraph{(4)}
In \cite{sethuraman_varadhan_yau_00}, the so-called \emph{graded sector condition} is formulated and Theorem KV is applied to tagged particle diffusion in general (non-zero mean) non-reversible exclusion processes, in $d\ge3$.

\paragraph{(5)}
For a more complete list of applications of Theorem KV together with the strong and graded sector conditions, see the surveys \cite{olla_01}, \cite{komorowski_landim_olla_09}.\\

Checking conditions \eqref{conditionA} and \eqref{conditionB} (or, equivalently, condition \eqref{conditionC}) in particular applications is typically not easy. In the applications to RWRE in \cite{toth_86}, the conditions were checked by some tricky bare hand computations. In \cite{varadhan_96}, respectively, \cite{sethuraman_varadhan_yau_00}, the so-called \emph{sector condition}, respectively, the \emph{graded sector condition} were introduced and checked for the respective models.

We recall from \cite{sethuraman_varadhan_yau_00} the graded sector condition. Assume that the Hilbert space $\cH=\cL^2(\Omega, \pi)$ is graded
\begin{equation}
\label{grading2}
\cH=\overline{\oplus_{n=0}^\infty\cH_n},
\end{equation}
and the infinitesimal generator is consistent with this grading in the sense of \eqref{grading}.

\begin{theorem*}[\bf SVY]
\label{thm:svy}
Assume that the Hilbert space and the infinitesimal generator
$G=-S+A$ are graded in the sense specified above and, in addition, there exist $\gamma\in[0,1)$ and  $C<\infty$ such that for any $n\in\N$ and any $g\in\cH_n$, $h\in\cH_{n+1}$
\begin{equation}
\label{gsc}
\abs{(h, A_+ g)}\le C n^{\gamma}\sqrt{(h,Sh)}\sqrt{(g,Sg)}.
\end{equation}
If $f\in\cH$ with $(f,\one)=0$ is such that
\begin{equation}
\label{H-1}
\norm{S^{-1/2}f}:=
\lim_{\lambda\to0}(f,u_\lambda)<\infty,
\end{equation}
then \eqref{conditionA} and \eqref{conditionB} hold and, as consequence, the conclusions of Theorem KV are valid.  \end{theorem*}


\section{Proof of the CLT}
\label{s:proof}

The proof of Theorem \ref{thm:clt} consists of three parts. In paragraph \ref{sss:lower_bound}, we prove diffusive lower bound on the variance of the displacement $X(t)$. We need this in order to exclude the possibility that the a priori martingale part of the displacement and the martingale approximation of the compensator in the limit just cancel out. (As it is well known, this happens for example in tagged particle diffusion in 1d simple symmetric exclusion process with nearest neighbour jumps, \cite{arratia_83}.) In paragraph \ref{sss:upper_bound}, we prove the $H_{-1}$-bound \eqref{H-1} for our particular case. Finally, in subsection \ref{ss:gsc}, we check conditions
\eqref{gsc} of Theorem SVY for our particular model.

\subsection{Diffusive bounds}
\label{ss:diffusive_bounds}


\subsubsection{Lower bound}
\label{sss:lower_bound}

For $s,t\in\R$ with $s<t$, let
\begin{equation}
\label{Mstdef}
M(s,t):=X(t)-X(s)-\int_s^t \varphi(\eta(u))\,\d
u=B(t)-B(s).
\end{equation}

\begin{lemma}
\label{lemma:forwbackw}
(1)
Fix $s\in\R$. The process $[s,\infty)\ni t\mapsto M(s,t)$ is a forward martingale with respect to the forward filtration
$\{\cF_{(-\infty,t]}:t\ge s\}$ of the process $t\mapsto\eta(t)$.
\\
(2)
Fix $t\in\R$. The process $(-\infty,t]\ni s\mapsto M(s,t)$ is a backward martingale with respect to the backward filtration $\{\cF_{[s,\infty)}:s\le t\}$ of the process $t\mapsto\eta(t)$.
\end{lemma}

\begin{proof}
There is nothing to prove about the first statement: the integral on the right-hand side of \eqref{Mstdef} was chosen exactly so that it compensates the conditional expectation of the infinitesimal increments of $X(t)$.

We turn to the second statement, which does need a proof. This consists of the following ingredients:
\begin{enumerate}[(1)]

\item
The displacements are reverted on the flipped backward trajectories $t\mapsto\wt\eta(t)$ defined in \eqref{revproc}:
\begin{equation}
\wt X(t)-\wt X(s)=-X(t)+X(s).
\end{equation}

\item
The forward process $t\mapsto\eta(t)$ and flipped backward process $t\mapsto\wt\eta(t)$ are identical in law (Yaglom reversibility).

\item
The function $\omega\mapsto\varphi(\omega)$ is odd with respect to the flip-map $\omega\mapsto -\omega$.

\end{enumerate}

\noindent
Putting these facts together (in this order), we obtain
\begin{align}
\lim_{h\to0}h^{-1}\condexpect{X(s-h)-X(s)}{\cF_{[s,\infty)}}
& =
\lim_{h\to0}h^{-1}\condexpect{-\wt X(-s+h)+\wt X(-s)}{\wt
\cF_{(-\infty,-s]}}\notag
\\[5pt]
\label{bwmart}
& =
-\varphi(\wt\eta(-s)) = \varphi(\wt\eta(s)).
\end{align}

\end{proof}

From Lemma \ref{lemma:forwbackw}, it follows directly that for any $s<t$, the random variables $M(s,t)$ and $\int_s^t \varphi(\eta(u))\,\d u$ are \emph{uncorrelated}, and therefore
\begin{align}
\label{variancesum}
\expect{(X(t)-X(s))^2}
&=
\expect{(M(s,t))^2}+
\expect{\big(\int_s^t \varphi(\eta(u))\,\d u\big)^2}
\\[5pt]
\notag
&=
(t-s) +
\expect{\big(\int_s^t \varphi(\eta(u))\,\d u\big)^2}.
\end{align}
Hence, the lower bound in \eqref{bounds}.


\subsubsection{Upper bound: $H_{-1}$-bound}
\label{sss:upper_bound}

We recall a general result proved in \cite{sethuraman_varadhan_yau_00}. See also the surveys \cite{olla_01}, \cite{komorowski_landim_olla_09} and further references cited therein.

Let $t\mapsto\xi(t)$ be the \emph{reversible} Markov process on the same state space $(\Omega, \pi)$ as the original $\eta(t)$ which has the infinitesimal generator $-S$.

\begin{lemma*}[\bf SVY]
\label{lemma:v}
Let $\varphi\in \cL^2(\Omega,\pi)$ with $\int \varphi\,\d\pi=0$. Then
\begin{equation}
\label{svy_bound}
\limsup_{t\to\infty}
t^{-1}\expect{\big(\int_0^t \varphi(\eta(s))\,\d s\big)^2}
\le
\lim_{t\to\infty}
t^{-1}\expect{\big(\int_0^t \varphi(\xi(s))\,\d s\big)^2}
=
2\Vert S^{-1/2} \varphi\Vert^2.
\end{equation}
\end{lemma*}

\medskip
\noindent
In our case,
\begin{equation}
\label{drsgen}
S=- \frac12 \Delta,
\end{equation}
and the reversible process $t\mapsto\xi(t)$ is the so-called \emph{diffusion in random scenery} process. That means:
\begin{equation}
\label{rwrs}
\xi(t):=\tau_{Z_t}\omega
\end{equation}
where $t\mapsto Z_t$ is a Brownian motion in $\R^d$ of covariance $\delta_{ij}$, independent of the field $\omega$. The function $\varphi:\Omega\to\R$ is $\varphi(\omega)=\omega(0)$. Thus, the upper bound in
\eqref{svy_bound} will be
\begin{equation}
\label{ourbound}
\lim_{t\to\infty}
t^{-1}\expect{\big(\int_0^t\varphi(\xi(s))\,\d s\big)^2} = \lim_{t\to\infty}
t^{-1}\expect{\big(\int_0^t\omega(Z_s)\,\d s\big)^2} =
\int_{\mathbb{R}^d} \abs{p}^{-2}\,\wh V(p)\,\d p.
\end{equation}
Here, the last step is straightforward computation with expectation taken over the Brownian motion $Z(t)$ \emph{and} over the random scenery $\omega$. The integral on the right-hand side is the same as in \eqref{sumcond}, and thus,
\eqref{ourbound} yields the upper bound in \eqref{bounds}. \hfill\qed


\subsection{Graded sector condition}
\label{ss:gsc}

As a first remark, note that condition \eqref{gsc} is equivalent to
\begin{equation}
\label{gsc2}
\norm{S^{-1/2}A_+S^{-1/2}\upharpoonright_{\cH_n}}\le Cn^{\gamma},
\end{equation}
where the operator $S^{-1/2}A_+S^{-1/2}\upharpoonright_{\cH_n}$ is meant as first defined on a dense subspace of $\cH_n$ and extended by continuity. In our case, the dense subspace will be $\wh \cC_n$ specified in \eqref{coredefin} and
\begin{equation}
S^{-1/2}A_+S^{-1/2}
=
\sum_{l=1}^d \abs{\Delta}^{-1/2} a^*_l \nabla_l\abs{\Delta}^{-1/2}
\end{equation}
The operators $\nabla_l\abs{\Delta}^{-1/2}$ map the subspaces $\wh \cC_n$ to themselves and are bounded, see \eqref{nonameopsonhKn}, \eqref{nonameopnorm}. In order to bound the norm of the operator $\abs{\Delta}^{-1/2} a^*_l:\cH_n\to\cH_{n+1}$, let $\wh u\in\wh \cC_{n}$, then
\begin{equation}
\abs{\Delta}^{-1/2} a^*_l \wh u(p_1,\dots,p_{n+1})
=
\frac{i}{\sqrt{n+1}}
\frac{1}{\abs{\sum_{m=1}^{n+1} p_m}}
\sum_{m=1}^{n+1}p_{ml}\wh u(p_1,\dots,p_{m-1},p_{m+1},\dots,p_n).
\end{equation}
Hence
\begin{align}
\label{bou}
&
(n+1)
\norm{\abs{\Delta}^{-1/2} a^*_l \wh u}^2
=
\\
\notag
&=
\int_{\R^d}\!\!...\!\!\int_{\R^d}
\frac1{\abs{\sum_{m=1}^{n+1} p_m}^2}
\abs{\sum_{m=1}^{n+1} p_{ml}\wh u(p_1,\!...,p_{m-1},p_{m+1},\!...,p_n)}^2
\prod_{m=1}^{n+1}\frac{\wh V(p_m)}{\abs{p_m}^2}
\d p_1\!...\d p_{n+1}
\\
\notag
&\le
(n+1)^2
\int_{\R^d}\!\!...\!\!\int_{\R^d}
\frac1{\abs{\sum_{m=1}^{n+1} p_m}^2}
\abs{p_{n+1,l}}^2
\abs{\wh u(p_1,\!...,p_{n})}^2
\prod_{m=1}^{n+1}\frac{\wh V(p_m)}{\abs{p_m}^2}
\d p_1\!...\d p_{n+1}
\\
\notag
&=
(n+1)^2
\int_{\R^d}\!\!...\!\!\int_{\R^d}
\abs{\wh u(p_1,\!\!...,p_{n})}^2
\prod_{m=1}^{n}\frac{\wh V(p_m)}{\abs{p_m}^2}
\left(\int_{\R^d}
\frac{p_{n+1,l}^2}{\abs{p_{n+1}}^2}
\frac{\wh V(p_{n+1})}{\abs{\sum_{m=1}^{n+1} p_m}^2}
\d p_{n+1}\right)
\d p_1\!...\d p_{n}.
\end{align}
In the second line, Schwarz's inequality and the symmetry of the function $\wh u(p_1,\dots,p_n)$ is used.

The innermost integral of the last expression in \eqref{bou} is bounded above by
\begin{align}
C^2:=
\sup_{p\in\R^d}\int_{\R^d}\frac{\wh V(p+q)}{\abs{q}^2} \,\d q
<\infty.
\end{align}
Thus, for $\wh u\in\wh \cC_n$
\begin{equation}
\norm{\abs{\Delta}^{-1/2} a^*_l \wh u}^2 \le C^2 (n+1)\norm{\wh u}^2 .
\end{equation}
Hence, by continuous extension,
\begin{equation}
\norm{\abs{\Delta}^{-1/2} a^*_l\upharpoonright_{\cH_n}} \le C \sqrt{n+1},
\end{equation}
and \eqref{gsc2} with $\gamma=1/2$ follows.
\qed

\vskip2cm \noindent {\bf Acknowledgement.} We thank Benedek Valkó for his remarks on the first draft of this paper. BT thanks the kind hospitality of the Mittag Leffler Insitute, Stockholm, where part of this work was done. The work of all authors was partially supported by OTKA (Hungarian National Research Fund) grant K 60708.

\end{document}